\documentclass{article}
\usepackage{amsmath,amssymb,amsthm,bm,enumerate,color,geometry,mathrsfs}
\usepackage[dvips]{graphicx}
\newcommand{\keywords}[1]{\textbf{Key words:} #1}
\newcommand{\msc}[1]{\textbf{MSC2010:} #1}

\def\rnum#1{\expandafter{\romannumeral #1}} 
\def\Rnum#1{\uppercase\expandafter{\romannumeral #1}} 
\title{Laws of the Iterated Logarithm for random walks on Random Conductance Models
  \footnote{This research was supported in part by JSPS KAKENHI Grant Number 25247007 and by 15J02838.}} 
\date{May 3, 2016}

\author{\textsc{Takashi Kumagai}\footnote{RIMS, Kyoto University, Kyoto
606-8502, Japan. e-mail: \texttt{kumagai@kurims.kyoto-u.ac.jp}}
          ~and \textsc{Chikara Nakamura}\footnote{RIMS, Kyoto University, Kyoto
606-8502, Japan. e-mail: \texttt{chikaran@kurims.kyoto-u.ac.jp}}}

\begin{document}
\maketitle

\newtheorem{Definition}{Definition}[section]
\newtheorem{prp}[Definition]{Proposition}
\newtheorem{thm}[Definition]{Theorem}
\newtheorem{ass}[Definition]{Assumption}
\newtheorem{lmm}[Definition]{Lemma}
\newtheorem{rmk}[Definition]{Remark}
\newtheorem{exa}[Definition]{Example}
\newtheorem{crl}[Definition]{Corollary}
\newtheorem{Notation}[Definition]{Notation}
\newtheorem{Application}[Definition]{Application}
\newtheorem{TenAss}[Definition]{Tentative Assumption}

\makeatletter
\@addtoreset{equation}{section}
\def\theequation{\thesection.\arabic{equation}}
\makeatother

\begin{flushleft}
 \keywords{Law of the iterated logarithm, Random conductance, Heat kernel}  \\
 \msc{60J10, 60J35.}  \\
\end{flushleft}

\begin{abstract}      
We derive laws of the iterated logarithm for random walks on random conductance models under the assumption that the random walks enjoy long time sub-Gaussian heat kernel estimates.  
\end{abstract}

\section{Introduction}
Random walks in random environments have been extensively studied for several decades
in probability and mathematical physics. 
Random conductance model (RCM) is a specific class in that random walks on the RCMs are reversible, and that the class includes many important examples.   
Recently, there has been significant progress in 
the study of asymptotic behaviors of random walks on RCMs. In particular, asymptotic behaviors such as
invariance principles and heat kernel estimates are obtained in the quenched sense, namely
almost surely with respect to the randomness of the environments, even for degenerate cases. 
One of the typical examples is the random walk on the supercritical percolation cluster on ${\mathbb Z}^d$. 
In this case, Barlow \cite{Barlow} obtained quenched long time Gaussian heat kernel estimates 
such as \eqref{UHK} and \eqref{LHK} below with $\alpha=d, \beta=2$. Soon after that, 
the quenched invariance principle was proved 
in \cite{SidSzn04} for $d\ge 4$ and later extended to all $d\ge 2$ in \cite{BB,MatPia07}. 
Namely, for a simple random walk  $\{Y^\omega_n\}_{n\ge 0}$ on the cluster, it was proved that 
$\varepsilon Y^\omega_{[t/\varepsilon^2]}$ converges as $\varepsilon\to 0$ to Brownian motion 
on ${\mathbb R}^d$ with covariance $\sigma^2 I$, $\sigma>0$, for almost all environment $\omega$. 
We note that the proof for $d\ge 3$ uses the heat kernel estimates given 
in \cite{Barlow}.

The RCM on a graph is a family of  non-negative random variables indexed by  edges of the graph. 
Supercritical bond percolation cluster is a typical (degenerate) RCM  which endows each edge of $\mathbb{Z}^d$ with 
     i.i.d.  Bernoulli random variable. 
The quenched Gaussian heat kernel estimates are established for various other RCMs, for example  
\begin{itemize}
 \item[(a)] uniformly elliptic conductances (\cite{Delmotte}),
 \item[(b)] i.i.d. unbounded conductances bounded from below by a strictly positive constant (\cite{BD}),  
 \item[(c)] i.i.d. conductances bounded from above and some tail condition near $0$ (\cite{BKM}),  
 \item[(d)] random walks on the level sets of Gaussian free fields and the framework of random interlacements (\cite{Sapozhnikov}),
 \item[(e)] positive conductances with some integrability condition (\cite{ADS}).  
\end{itemize}
Note that conductances in (a), (d), (e) are not necessarily i.i.d..  
Note also that, while (b)-(d) are discussed on $\mathbb{Z}^d$, (a) and (e) are discussed for more general 
graphs with some analytic properties. 
Quenched invariance principles for the random walks on RCMs are also established extensively.  
For more details, see \cite{Biskup,Kumagai} and the references therein. 

We are interested in further quenched asymptotic behaviors of the random walks on RCMs. 
The aim of this paper is to establish the laws of the iterated logarithm (LILs) for the
sample paths of the random walk such as \eqref{LIL10} and \eqref{LIL20} below in the quenched level. 
In fact, for the random walk on the supercritical percolation cluster,  
Duminil-Copin \cite{Copin} obtained the standard LIL (limsup version as in \eqref{LIL10}) by using the results of \cite{Barlow}. 
Also, in \cite{Kubota} the LIL is obtained for a class of transient random walk in random environments. 
The novelty of this paper is twofold.
\begin{itemize}
       \item We establish another law of the iterated logarithm (liminf version as in \eqref{LIL20}).
 
       \item  We establish quenched LILs for random walks on much more general RCMs.  
\end{itemize}
Our approach is through the heat kernel estimates. Namely, we assume the quenched heat kernel estimates 
(Assumption \ref{Ass}) and establish the quenched LILs (Theorem \ref{Thm10}). 
Since the quenched heat kernel estimates are established for many RCMs, our theorem applies for 
those examples as we discuss in Section \ref{Sec1-2}. 

The organization of the paper is as follows. We first explain the framework and main results of this paper. 
In Section \ref{Conseq}, we give the preliminary estimates to prove the main results. 
In Section \ref{Sec:SupLIL} we prove the LIL and in Section \ref{Sec:InfLIL} we prove another LIL.
Finally in Section \ref{Zd}, we assume the ergodicity of the media when $G=\mathbb{Z}^d$ and 
prove that the constants appearing in the limsup and liminf in the LILs are deterministic.

\subsection{Framework and main results}

Let $G=(V,E)$ be the countably infinite, locally finite and connected graph. 
We can define the graph distance $d: V \times V \to [0,\infty )$ in the usual way, i.e. the shortest length of path in $G$.
Write $B(x,r) = \{ y \in V(G) \mid d(x,y) \le r \}$.   
Throughout this paper we assume that there exist $\alpha \ge 1$ and $c_1, c_2 >0$ such that 
                 \begin{align}\label{vols}
                              c_1 r^{\alpha} \le \sharp B(x,r) \le c_2 r^{\alpha}
                 \end{align}
   holds for all $x \in V(G)$ and $r \ge 1$.

We assume that the graph $G$ is endowed with the non-negative weights (or conductance)  $\omega = \{ \omega (e) \mid e \in E \}$ 
 which are defined on a probability space  $(\Omega , \mathcal{F} , \mathbb{P} )$. 
We write $\omega(e) = \omega_e = \omega_{xy}$ if $e = xy$. 
 We take the base point $x_0$ of $G$ and set 
 $V (G^{\omega} ) = \{ v \in V(G) \mid x_0 \overset{ \omega }{\longleftrightarrow} v \}$, where $x_0 \overset{ \omega }{\longleftrightarrow} v$ means that 
      there exists a path  $\gamma = e_1 e_2 \cdots e_k $ from $x_0$ to $v$ such that $\omega (e_i) > 0$ for all $i=1,2, \cdots , k$. 
 We also define $\mathcal{C} (\omega )$ as the set of all vertices $x$ which satisfy $x \overset{\omega}{\longleftrightarrow} \infty$, i.e. there exists an infinite length and self-avoiding path 
   $\gamma = e_1 e_2 \cdots$ starting at $x$ which satisfies $\omega (e_i) >0$ for all $i$.  
Note that if each weight $\omega (e)$ is strictly positive, then $V(G^{\omega}) = \mathcal{C} (\omega ) = V(G)$. 
Let $\mu^{\omega} (x) = \sum_{y;y \sim x} \omega_{xy}$ be the weight of $x$, $\displaystyle V^{\omega} (A) = \sum_{y \in A \cap V(G^{\omega}) } \mu^{\omega} (y)$  
 be the volume of $A \subset V(G)$  and  $\displaystyle V^{\omega}(x,r) = V^{\omega} ( B(x,r) )$ be the volume of the ball $B(x,r)$. 
  We also denote  $B^{\omega} (x,r) = B(x,r) \cap V(G^{\omega})$.

Next we define 
the random walk on the weighted graph. 
Let $\{ X_n^{\omega} \}_{n \ge 0}$ be the discrete time random walk on $V(G^{\omega})$ 
 whose transition probability is given by $\displaystyle P^{\omega}(x,y) = \frac{ \omega_{xy} }{ \mu^{\omega} (x) }$.  
We write $P_n^{\omega} (x,y) = P_x^{\omega} ( X_n^{\omega} = y)$. 
The heat kernel is denoted by $\displaystyle p_n^{\omega} (x,y) = \frac{P_n^{\omega} (x,y)}{ \mu^{\omega} (y) }$.

For our main results, we assume the following conditions. 
Note that $\alpha \ge 1$ is the same as in \eqref{vols}.    
\begin{ass}  \label{Ass}
        There exist $\Omega_0 \in \mathcal{F}$ with $\mathbb{P} (\Omega_0) = 1$, positive constants $c_{1.1} , c_{1.2} , \cdots , c_{1.6}, \beta , \epsilon$, with $\epsilon +1< \beta$ 
        and random variables $N_{x,\epsilon} (\omega )$ $(x \in V(G), \omega \in \Omega_0)$ such that the following hold.
                    \begin{enumerate}
                              \item[(1)]  For all $\omega \in \Omega_0$, $x \in V(G^{\omega})$ and $r \ge N_{x,\epsilon} (\omega)$,
                              it holds that  
                                            \begin{align} \label{Vol}
                                                       c_{1.1} r^{\alpha}  \le V^{\omega} (x,r)  \le c_{1.2} r^{\alpha}.
                                            \end{align}
                              
                              \item[(2)] For all $\omega \in \Omega_0$, $\{ X_n^{\omega} \}_{n \ge 0}$ enjoys the following heat kernel estimates;
                                        \begin{align}  \label{UHK}
                                              p_n^{\omega} (x,y) \le \frac{c_{1.3}}{n^{\alpha / \beta}}  \exp \left[- c_{1.4} \left( \frac{d(x,y)}{n^{1/\beta} } \right)^{\beta/(\beta -1)}  \right]
                                        \end{align}
                                  for $d(x,y) \vee N_{x,\epsilon} (\omega) \le n$, and 
                                        \begin{align}   \label{LHK}
                                              p_n^{\omega} (x,y) + p_{n+1}^{\omega} (x,y)  \ge 
                                                          \frac{c_{1.5}}{n^{\alpha / \beta}}  \exp \left[- c_{1.6} \left( \frac{d(x,y)}{n^{1/\beta} } \right)^{\beta/(\beta -1)}  \right]
                                        \end{align}
                                   for $d(x,y)^{1+\epsilon} \vee N_{x,\epsilon} (\omega) \le n$.

                              \item[(3)]  There exists a non-increasing function $f_{\epsilon} (n)$ which satisfies 
                                        \begin{align}  \label{Tail}
                                                 \mathbb{P} (N_{x,\epsilon}  \ge n) \le f_{\epsilon} (n) ~~~ \text{and} ~~~
                                                   \sum_{n \ge 1} n^{\alpha \beta} f_{\epsilon} (n) < \infty .
                                         \end{align}
                                  
                    \end{enumerate}
   \end{ass}

Now we state the main result of this paper. 
  \begin{thm}   \label{Thm10}
        Suppose that Assumption \ref{Ass} holds. Then for almost all environment $\omega \in \Omega$ 
         there exist positive constants $C_1 = C_1 (\omega)$ and $C_2 = C_2 (\omega)$ such that the following hold. 
                        \begin{align}
                                     \limsup_{n \to \infty}   \frac{d(X_0^{\omega}, X_n^{\omega}) }{ n^{1/\beta} (\log \log n )^{1-1/\beta} } = C_1,
                                                    \qquad   \text{$P_x^{\omega}$-a.s. for all $x \in V(G^{\omega})$},  \label{LIL10} \\ 
                                     \liminf_{n \to \infty}   \frac{ \max_{0\le \ell \le n} d(X_0^{\omega}, X_{\ell}^{\omega}) }{ n^{1/\beta} (\log \log n )^{-1/\beta} } = C_2,
                                                    \qquad   \text{$P_x^{\omega}$-a.s. for all $x \in V(G^{\omega})$}.   \label{LIL20} 
                        \end{align} 
  \end{thm}

We note that we can replace $d(X_0^{\omega}, X_n^{\omega})$ in \eqref{LIL10} to $\displaystyle \max_{0 \le \ell \le n} d(X_0^{\omega}, X_{\ell}^{\omega})$ with possibly different $C_1$.    
We also note that if the random walk can be embedded into Brownian motion in some strong sense (which seems plausible in various concrete models), then \eqref{LIL10},\eqref{LIL20} can be shown as a consequence (\cite{Biskper}). It would be very interesting to prove such a strong approximation theorem. 

\medskip

The constants $C_i$ above may depend on the environment $\omega$. In order to guarantee that
they are deterministic constants, we need to assume the ergodicity of the media. For the purpose, 
we now consider the case $G=\mathbb{Z}^d$. In this case, we can define the shift operators $\tau_{x}: \Omega \to \Omega$ $(x \in \mathbb{Z}^d)$ as 
             \begin{align*}
                                (\tau_{x} \omega)_{yz} = \omega_{y+x, z+x}.
             \end{align*}
We assume the following ergodicity of the media. 

    \begin{ass}  \label{Ass2}
         Assume that $(\Omega, \mathcal{F},\mathbb{P})$ satisfies the following conditions;
                  \begin{enumerate}
                            \item[(1)]       $\mathbb{P}$ is ergodic with respect to the translation operators $\tau_x$, i.e. $\mathbb{P} \circ \tau_x = \mathbb{P}$ and 
                                                 for any $A \in \mathcal{F}$ with $\tau_x (A) = A$ for all $x \in \mathbb{Z}^d$ then $\mathbb{P} (A) = 0$ or $1$. 
                            
                             \item[(2)]  For almost all environment $\omega$, $\mathcal{C} (\omega)$ contains an unique infinite connected component. 
                 \end{enumerate}

    \end{ass}

\begin{thm}  \label{Thm20}
     Suppose that Assumption \ref{Ass} and Assumption \ref{Ass2} hold. 
     Then we can take $C_1, C_2$ in Theorem \ref{Thm10} as deterministic constants (which do not depend on $\omega$).
 \end{thm} 

\begin{rmk}
In this paper, we only consider discrete time Markov chains, but similar results hold for continuous time Markov chains (constant speed random walks and variable speed random walks); see \cite{Nak}.
\end{rmk}

\subsection{Examples}\label{Sec1-2}
In this subsection, we give examples for which our results hold. 

\begin{exa}[Bernoulli supercritical percolation cluster]
Barlow \cite[Theorem 1]{Barlow} proved that heat kernels of simple random walks on the super-critical percolation cluster for $\mathbb{Z}^d$, $d\ge 2$   
 satisfy Assumption \ref{Ass} with $\alpha = d$, $\beta = 2$ and $f_{\epsilon} (n) = c \exp (-c^{\prime} n^{\delta})$ for some $c,c^{\prime},  \delta>0$. 
(In \cite{Barlow}, heat kernels for continuous time random walk were obtained. 
See the remark after \cite[Theorem 1]{Barlow} and \cite[Section A]{BB} for discrete time modifications.)   
Since the media is i.i.d. and there exists an unique infinite connected component, we can obtain the LILs \eqref{LIL10} and \eqref{LIL20} with deterministic constants. 
Note that \eqref{LIL10} for the supercritical percolation cluster was already obtained by \cite[Theorem 1.1]{Copin}. 
\end{exa}

\begin{exa}[Uniform elliptic case] Suppose the graph $G=(V,E)$ endowed with weight $1$ on each edge
satisfies \eqref{vols} and the scaled Poincar\'e inequalities.  
Put random conductance on each edge so that 
$c_1\le \omega (e) \le c_2$ for all $e\in E$ and for almost all $\omega$,
where $c_1,c_2>0$ are deterministic constants. Then 
Assumption \ref{Ass} holds with $\beta = 2$
and $N_{x,\epsilon} \equiv 1$. So the LILs \eqref{LIL10} and \eqref{LIL20} hold. 
\end{exa}

\begin{exa}[Gaussian free fields and random  interlacements] 
Sapozhnikov  \cite[Theorem 1.15]{Sapozhnikov} proved that for 
$\mathbb{Z}^d$, $d\ge 3$, 
 the random walks on (i) certain level sets of Gaussian free fields;  
(ii) random interlacements at level $u>0$; (iii) vacant sets of random interlacements for suitable level sets, satisfy our Assumption \ref{Ass}  
with $\alpha = d$, $\beta = 2$ and the tail estimates of $N_{x,\epsilon} (\omega)$ as 
$f_{\epsilon} (n) = c\exp (-c^{\prime} (\log n)^{1+\delta} )$ for some $c,c^{\prime}, \delta >0$.   
This subexponential tail estimate is sufficient for Assumption \ref{Ass} $(3)$. 
Since the media is ergodic and 
there is an unique infinite connected components (see \cite{RS}, \cite[Corollary 2.3]{Sznitman} and \cite[Theorem 1.1]{Teixeira}), 
the LILs \eqref{LIL10} and \eqref{LIL20} hold with deterministic constants. 
\end{exa}

\begin{exa}[Uniform elliptic RCM on fractals]
Let $a_1=(0,0), a_2=(1,0),$ $a_3=(1/2,\sqrt 3/2)$, $I=\{1,2,3\}$ and set
$F_i(x)=(x-a_i)/2+a_i$ for $i\in I$. Define
\[
V=\bigcup_{n\in \mathbb N}\Big(2^n \bigcup_{i,i_1,\cdots, i_n\in I}F_{i_n}\circ \cdots \circ F_{i_1}(a_i)\Big),~~
E=\bigcup_{n\in \mathbb N}\Big(2^n \bigcup_{i_1,\cdots, i_n\in I}F_{i_n}\circ \cdots \circ F_{i_1}(B_0)\Big),
\]
where $B_0=\{\{x,y\}: x\ne y\in \{a_1,a_2,a_3\}\}$. $G=(V,E)$ is called 
the 2-dimensional pre-Sierpinski gasket. Put random conductance on each edge so that 
$c_1\le \omega (e) \le c_2$ for all $e\in E$ and almost all $\omega$,
where $c_1,c_2>0$ are deterministic constants. Then 
Assumption \ref{Ass} holds with $\alpha = \log 3/\log 2$, $\beta = \log 5/\log 2>2$
and $N_{x,\epsilon}\equiv 1$. (In fact, this can be generalized to the uniform finitely 
ramified graphs for some $\alpha\ge 1$ and $\beta\ge 2$; see \cite{HK}.) 
So the LILs \eqref{LIL10} and \eqref{LIL20} hold. 
\end{exa}

We note that among the examples mentioned at the beginning of this paper, (b), (c) and (e) are for continuous time Markov chains, so the LILs will be discussed in  \cite{Nak}.    

\section{Consequences of Assumption \ref{Ass}}  \label{Conseq}
In this section, we prepare the preliminary results of Assumption \ref{Ass}.

\subsection{Consequences of heat kernel estimates}
We first give consequences  of the heat kernel estimates \eqref{UHK} and \eqref{LHK}.  

\begin{lmm}   \label{REHK20} 

  \begin{enumerate}   \renewcommand{\labelenumi}{(\arabic{enumi}).}
         \item[(1)]  
               There exist $c_1 ,c_2 >0$ such that for almost all $\omega \in \Omega$, 
                           \begin{gather*}   
                                            P_y^{\omega}  \left( \max_{0 \le j \le n} d(x ,X_j^{\omega} ) \ge 3r \right)
                                                           \le c_{1} \exp \left( -c_{2} \left( \frac{r^{\beta}}{n} \right)^{\frac{1}{\beta -1}} \right)
                             \end{gather*}
               holds   for all $n \ge 1, r  \ge 1$ and  $x,y \in V(G^{\omega})$ with $\displaystyle \max_{z \in B(y,2r) } N_{z,\epsilon} (\omega ) \le r $ 
                  and $d(x,y) \le r$.  
  
         \item[(2)]  
              There exist $c_3 ,c_4, R_0 >0$ such that for almost all $\omega \in \Omega$, 
                      \begin{equation*}   
                                       P_x^{\omega}  \left( \max_{0 \le j \le n } d(X_0^{\omega}, X_j^{\omega} ) \le r \right)  \le c_3\exp \left( -c_4 \frac{n}{r^{\beta}} \right)    
                      \end{equation*}
                holds  for all $n \ge 1, r \ge  R_0$ and $x \in V(G^{\omega})$ with $\displaystyle  \max_{y \in B(x,r)}  N_{y, \epsilon} (\omega ) \le 2r$.
      
         \item[(3)]  
              Suppose $\epsilon +1 < \beta$.  Then there exist $ c_5, c_6>0$ and $\eta \ge 1$ such that for almost all $\omega \in \Omega$, 
                          \begin{gather*}      
                                        P_x^{\omega}  \left( \max_{0 \le j \le n } d(X_0^{\omega},X_j^{\omega}) \le r \right)  
                                                    \ge c_5\exp \left( -c_6 \frac{n}{r^{\beta}} \right) 
                         \end{gather*}
                holds for all $x \in V(G^{\omega})$ and $n \ge 1, r \ge 1$ with $\displaystyle \max_{z \in B(x, 3\eta r) } N_{z ,\epsilon} (\omega ) \le r^{1/\beta}$.

    \end{enumerate}
\end{lmm}

Since the computations are standard, we omit the proof. Indeed, 
(1) can be proved by simple modifications of \cite[Lemma 3.9]{Barlow1}, and 
(2) can be proved similarly to \cite[Lemma 3.2]{KN}. 
(3) is simple modification of \cite[Proposition 3.3]{KN} respectively.

 Let $c_5,c_6>0$ be  as in Lemma \ref{REHK20}\,(3). 
 Define $a_k, b_k, \lambda_k, u_k, \sigma_k$ 
as follows:
           \begin{equation}  \label{Not10}
                        a_k^{\beta}  = e^{k^2} , ~   
                        b_k^{\beta} = e^k , ~  
                        \lambda_k =c_6^{-1}\log (c_5(1+k)^{2/3}) ,~   
                        u_k = \lambda_k a_k^{\beta} ,  ~
                        \sigma_k = \sum_{i=1}^{k-1} u_i.
           \end{equation}
 
\begin{crl}[Corollary of Lemma \ref{REHK20}\,(3)]    \label{REHK90}
 Let $\eta \ge 1$ be as in Lemma \ref{REHK20}\,(3). 
 Then the following holds for almost all $\omega \in \Omega$, all $x \in V(G^{\omega})$ and $k \ge 1$ with 
  $\displaystyle \max_{z \in B(x, 4\eta a_k )} N_{z,\epsilon} (\omega ) \le a_k^{1/\beta}$, 
                         \begin{align*} 
                  \min_{z \in B^{\omega} (x,a_k ) }  
  P_z^{\omega }  \left(  \max_{0 \le s \le u_k} d(X_0^{\omega},X_s^{\omega} )  \le a_k  \right)  \ge 
                                    \frac{1}{ (1+k)^{2/3} }. 
                         \end{align*}
\end{crl}

The heat kernel estimates \eqref{UHK} and \eqref{LHK} also give the triviality of tail events. 
\begin{thm}[$0-1$ law for tail events]  \label{RE0-1}  
For almost all $\omega \in \Omega$, the following holds;  
Let $A^{\omega}$ be a tail event, i.e. $\displaystyle  A^{\omega} \in \bigcap_{n=0}^{ \infty }  \sigma \{ X_k^{\omega} : k \ge n \} $.
Then either  $P_x^{\omega} (A^{\omega})  = 0 $ for all $x $ or $P_x^{\omega} (A^{\omega}) = 1$ for all $x$ holds. 
\end{thm}

The proof of Theorem \ref{RE0-1} is quite similar to that of \cite[Proposition 2.3]{BK}, so we omit the proof.

\subsection{Consequences of the tail estimate \eqref{Tail} }

We next give simple consequences of the tail estimate \eqref{Tail}.
Recall the notations in \eqref{Not10}, and set $\Phi (q) = q^{1/\beta} (\log \log q)^{1-1/\beta}$.
\begin{lmm}  \label{Tail10}
     \begin{enumerate}   \renewcommand{\labelenumi}{(\roman{enumi})}
         \item[(1)]  
            Suppose that $f_{\epsilon} (n)$ satisfies $\displaystyle  \sum_{n} n^{\alpha} f_{\epsilon} (n) < \infty$. 
                   Then for any $\gamma_1, \gamma_2>0$ and for almost all $\omega \in \Omega$, 
                   there exists $L_{x,\epsilon, \gamma_1, \gamma_2} (\omega)>0$ such that the following hold for all $n \ge L_{x,\epsilon, \gamma_1, \gamma_2} (\omega)$,
                              \begin{eqnarray*} 
                                              \gamma_1 a_n \ge   \max_{z \in B(x, \gamma_2 a_n ) }  N_{z, \epsilon} (\omega ),~~  
                                              \gamma_1 b_n \ge   \max_{z \in B(x, \gamma_2 b_n ) }  N_{z, \epsilon} (\omega ). 
                             \end{eqnarray*}

         \item[(2)]   
              Suppose that $f_{\epsilon} (n)$ satisfies $\displaystyle  \sum_{n} n^{\alpha} f_{\epsilon} (n) < \infty$. 
                   Then for any $\gamma_1, \gamma_2>0$, $q > 1$ and for almost all $\omega \in \Omega$, 
                   there exists $L_{x,\epsilon, \gamma_1, \gamma_2, q} (\omega)>0$ such that the following hold for all $n \ge L_{x,\epsilon, \gamma_1, \gamma_2, q} (\omega)$,
                              \begin{eqnarray*} ~~~~
                                             \gamma_1 \Phi (q^n) \ge \max_{z \in B(x,\gamma_2 \Phi (q^n)) }  N_{z,\epsilon} (\omega ),~~   
                                             \gamma_1 q^{(n-1)/\beta} \ge   \max_{z \in B(x, \gamma_2 q^{(n-1)/\beta} )} N_{z, \epsilon} (\omega ) .         
                             \end{eqnarray*}

         \item[(3)]   
          Suppose that $f_{\epsilon} (n)$ satisfies $\displaystyle  \sum_{n} n^{\alpha \beta} f_{\epsilon} (n) < \infty$. 
                    Then for all $\gamma_1, \gamma_2 > 0$ and for almost all $\omega \in \Omega$, 
                    there exists $K_{x,\epsilon, \gamma_1, \gamma_2} (\omega)>0$ such that the following holds for all 
                    $n \ge K_{x,\epsilon, \gamma_1, \gamma_2} (\omega)$,
                                  \begin{align*} 
                                                 \gamma_1 a_n^{1/\beta} \ge \max_{z \in B(x,\gamma_2 a_n) }  N_{z,\epsilon} (\omega ).
                                 \end{align*}
      \end{enumerate}
\end{lmm}

\begin{proof}
We only prove the first inequality in (1).  
It is easy to see that 
             \begin{align*}
                           \mathbb{P} \left(  \max_{z \in B(x,\gamma_2 n)} N_{z,\epsilon} > \gamma_1 n\right)  
                                \le \sum_{z \in B(x,\gamma_2 n)}  \mathbb{P} \left( N_{z,\epsilon} \ge \gamma_1 n \right) 
                          \le c_1 (\gamma_2 n)^{\alpha} f_{\epsilon} (\gamma_1 n).    
               \end{align*} 
The assumption implies $\sum_{n} n^{\alpha} f_{\epsilon} (\gamma_1 n) < \infty$, so the conclusion follows by the Borel-Cantelli Lemma.  
\end{proof}

\section{Proof of LIL}  \label{Sec:SupLIL} 
In this section, we prove  \eqref{LIL10} in Theorem \ref{Thm10}. 
We continue to use the notation $\Phi (q) = q^{1/\beta} (\log \log q)^{1-1/\beta}$ in this section.

\begin{thm}  \label{Uppsup}
Suppose that Assumption \ref{Ass} holds. 
Then there exists $c_+>0$ such that the following holds 
for almost all $\omega \in \Omega$, 
                  \begin{align*}
                                \limsup_{n \to \infty}  
                                  \frac{ \max_{0 \le k \le n}  d(X_0^{\omega}, X_k^{\omega} ) }{ n^{1/\beta} (\log \log n)^{1-1/\beta}  }   \le c_+,
                                  \qquad   \text{$P_x^{\omega}$-a.s. for all $x \in V(G^{\omega})$. } 
                  \end{align*}
\end{thm}
  
\begin{proof}
By Lemma \ref{REHK20}\,(1) we have   
  \begin{align*}
          &  P_x^{\omega } \left( \max_{0\le k \le q^n}  d(X_0^{\omega} , X_k^{\omega} ) \ge \eta \Phi (q^n) \right)  
                      \le c_1 \exp \left[ -c_2 \left( \frac{ ( \eta \Phi (q^n)) ^{\beta} }{ q^n } \right)^{\frac{1}{\beta - 1} }  \right]   \\
          &= c_1 \exp \left[ - c_2 \eta^{\beta / (\beta -1)}  \log \log q^n  \right]   =  c_1 \left( \frac{1}{n \log q} \right)^{c_2 \eta^{\beta /( \beta - 1) } }
  \end{align*} 
 for all $q \ge 1$, almost all $\omega$ and  $n$ with $\displaystyle \max_{z \in B(x, 2 \Phi (q^n)) }  N_{z,\epsilon} (\omega) \le \Phi (q^n)$.  
 Therefore the above estimate holds for $n \ge L_{x,\epsilon,1,2,q} (\omega)$ by Lemma \ref{Tail10}\,(2).   
   
 So taking $\eta >0$ large enough and using the  Borel-Cantelli Lemma,  we have
        \begin{align*}
                \limsup_{n \to \infty}   
                \frac{ \max_{0 \le k \le q^n } 
                d (X_0^{\omega}, X_{k}^{\omega } ) }{ \Phi (q^n) }  \le \eta.
       \end{align*} 
 We can easily obtain the conclusion from the above inequality. 
\end{proof}

\begin{thm}  \label{Lowsup}
Suppose that Assumption \ref{Ass} holds. 
Then there exists $c_->0$ such that the following holds for almost all $\omega \in \Omega$, 
                  \begin{align*}
                                \limsup_{n \to \infty}  \frac{ d(X_0^{\omega}, X_n^{\omega} ) }{ n^{1/\beta} (\log \log n)^{1-1/\beta}  }   \ge c_-,   
                                            \qquad   \text{$P_x^{\omega}$-a.s. for all $x \in V(G^{\omega})$}.
                  \end{align*}
\end{thm}

\begin{proof}  
 Note that $\displaystyle d(X_0^{\omega}, X_{q^n}^{\omega} ) \ge d(X_{q^{n-1}}^{\omega} , X_{q^n}^{\omega} ) - d(X_0^{\omega} , X_{q^{n-1}}^{\omega} )$ for any $q > 1$.  
 By Theorem \ref{Uppsup}, for almost all $\omega \in \Omega$ and $P_x^{\omega}$-a.s. there exists a constant  $M_x$ such that
                        \begin{align*}
                                     \frac{d(X_0^{\omega}, X_{q^{n-1}}^{\omega}) }{ \Phi (q^n)}  
                                        =   \frac{d(X_0^{\omega}, X_{q^{n-1}}^{\omega}) }{ \Phi (q^{n-1})}    \frac{\Phi (q^{n-1})}{ \Phi (q^{n})}      \le \frac{ 2c_+ }{q^{1/\beta}} 
                        \end{align*}
  holds for any $n \ge M_x $, where $c_+$ is as in Theorem 3.1.  The right hand side of the above inequality can be small enough by taking $q$ sufficiently large. 
  So it is enough to show that there exists a positive constant $c_{-}$ independent of $q$ such that the following holds, 
                          \begin{align}   \label{RELILLow01}
                                      \limsup_{n \to \infty}   \frac{d(X_{q^{n-1}}^{\omega}, X_{q^n}^{\omega} ) }{ \Phi (q^n)}  \ge c_-.
                          \end{align}

 We may and do take $q \ge 2$. 
 To prove \eqref{RELILLow01}, let $\mathcal{F}_n^{\omega} = \sigma \left( X_k^{\omega} \mid k \le n \right)$ and  $t_n = q^n - q^{n-1}$. 
Set $\kappa>0$ so that $c_{1.1} \kappa^{\alpha}-c_{1.2} \ge 1$. 
Let $\lambda>0$ be a small constant so that $\kappa \lambda < 1$. 
By Theorem \ref{Uppsup} there exists a constant $c_+^{\prime}$ such that $d(X_0^{\omega}, X_{q^{n-1}}^{\omega}) \le c_+^{\prime} \Phi (q^{n-1})$ for almost all $\omega$ and for sufficiently large $n$.  
We first note that 
                  \begin{align}     \label{RELILLow03}
                       &P_x^{\omega} \left(  d(X_{q^{n-1} }^{\omega} , X_{q^n}^{\omega}) 
                                 \ge \lambda \Phi (q^n) \middle| \mathcal{F}_{q^{n-1}}^{\omega}  \right)  \notag \\
                       &\ge P_x^{\omega} \left(  d(X_{q^{n-1} }^{\omega} , X_{q^n}^{\omega}) \ge \lambda \Phi (q^n) ,
                           d(X_0^{\omega}, X_{q^{n-1}}^{\omega}) \le c_+^{\prime} \Phi (q^{n-1}) \middle| \mathcal{F}_{q^{n-1}}^{\omega}  \right) \notag  \\ 
                       &=  1_{ \left\{  d(X_{0}^{\omega} , X_{q^{n-1}}^{\omega}) \le c_+^{\prime} \Phi (q^{n-1}) \right\} } 
                            P_{X_{q^{n-1}}^{\omega} }^{\omega}  \left(  d(X_{0}^{\omega} , X_{t_n}^{\omega}) \ge \lambda \Phi (q^n) \right)   \notag   \\
                       &\ge \left( \min_{y \in B^{\omega}(x,  c_+^{\prime} \Phi (q^{n-1})) } 
                                     P_y^{\omega} \left( d(X_{0}^{\omega} , X_{t_n}^{\omega}) \ge \lambda \Phi (q^n) \right) \right) 
                                     1_{ \left\{  d(X_{0}^{\omega} , X_{q^{n-1}}^{\omega}) \le c_+^{\prime} \Phi (q^{n-1}) \right\} }  .
                  \end{align}
We estimate the first term of  \eqref{RELILLow03}. 
For any $n$ with $\lambda \Phi (q^n) \ge N_{y, \epsilon} (\omega) $, using \eqref{Vol} we have 
                        \begin{align*}
                                    \mu^{\omega} (B(y,\kappa \lambda \Phi (q^n) ) \setminus B(y,\lambda \Phi (q^n) )) 
                                                  \ge c_{1.1} (\kappa \lambda \Phi (q^n) )^{\alpha} - c_{1.2} (\lambda \Phi (q^n) )^{\alpha} \ge (\lambda \Phi (q^n) )^{\alpha}.
                        \end{align*}
 So for such $n$ and for $y \in B^{\omega}(x, c_{+}^{\prime} \Phi (q^{n-1}))$ we have 
              \begin{align*}   \label{RELILLow05}
                      &P_y^{\omega} \left( \lambda \Phi (q^n) \le d(X_0^{\omega}, X_{t_n}^{\omega})  \le \kappa \lambda \Phi (q^n) \right)   
                            \ge \sum_{z \in B^{\omega} (y,\kappa \lambda \Phi (q^n) ) \setminus B^{\omega} (y,\lambda \Phi (q^n) ) } p_{t_n}^{\omega} (y,z)  \mu^{\omega} (z) \notag   \\
                      & \ge \frac{c_{1.5}}{t_n^{\alpha / \beta} }  \exp \left[ - c_{1.6} \left( \frac{ (\kappa \lambda \Phi (q^n))^{\beta} }{ t_n} \right)^{\frac{1}{\beta -1}} \right]
                               \mu^{\omega} \left( B(y,\kappa \lambda \Phi (q^n) ) \setminus B(y,\lambda \Phi (q^n) )  \right)  \notag  \\
                      &  \ge c_1 \left( \frac{1}{n} \right)^{c_2 
                     (\kappa\lambda)^{\beta/(\beta -1)} },
              \end{align*} 
  where we can take $c_1, c_2$ as the constants which do not depend on $q$.  
 Therefore for any $n$ with $\displaystyle  \max_{y \in B(x, c_+^{\prime} \Phi (q^{n-1})) }$ $ N_{y,\epsilon} (\omega) \le \lambda \Phi (q^n)$ we have 
                       \begin{align*}  
                                \min_{y \in B^{\omega} (x, c_+^{\prime} \Phi (q^{n-1})) }  
                                   P_y^{\omega} \left(  d(X_{0}^{\omega} , X_{t_n}^{\omega}) \ge \lambda \Phi (q^n) \right)  
                                         \ge c_1 \left( \frac{1}{n} \right)^{c_2 
         (\kappa\lambda)^{\beta/(\beta -1)} }.
                       \end{align*} 
 By Lemma \ref{Tail10}\,(2),  $\displaystyle  \max_{y \in B(x, c_+^{\prime} \Phi (q^{n-1}))  }  N_{y,\epsilon} (\omega) \le \lambda \Phi (q^n)$ 
     for all $n \ge L_{x,\epsilon, \lambda ,c_+^{\prime},q} (\omega)$. 
 As we mentioned before,  $\displaystyle    d(X_{0}^{\omega} , X_{q^{n-1}}^{\omega}) \le c_+^{\prime} \Phi (q^{n-1}) $ for sufficiently large $n$. 
 Thus for sufficiently small $\lambda$ we have   
                  \begin{align*}
                              &\sum_n P_x^{\omega} \left(  d(X_{q^{n-1} }^{\omega} , X_{q^n}^{\omega}) \ge \lambda \Phi (q^n) \middle| \mathcal{F}_{q^{n-1}}^{\omega} \right) 
                                = \infty .
                  \end{align*} 
  Hence  by the second Borel-Cantelli lemma, we have 
                       \begin{align*}
                                 \limsup_{n\to \infty}  \frac{d(X_{q^{n-1}}^{\omega}, X_{q^n}^{\omega} ) }{ \Phi (q^n) } \ge \lambda .
                       \end{align*} 
 We thus complete the proof. 
\end{proof}

By Theorem \ref{RE0-1}, Theorem \ref{Uppsup} and Theorem \ref{Lowsup}, we complete the proof of \eqref{LIL10} in Theorem \ref{Thm10}.

\section{Proof of another LIL}  \label{Sec:InfLIL}
In this section, we prove \eqref{LIL20} of Theorem \ref{Thm10}. 
\begin{thm}    \label{Liminf}
Suppose that Assumption \ref{Ass} holds. 
Then for almost all $\omega \in \Omega$  there exists 
$c = c(\omega)>0$ such that the following holds, 
    \begin{equation} \label{REinf10}
               \liminf_{n \to \infty } \frac{\max_{0 < \ell \le n} d(X_0^{\omega} ,X_{\ell}^{\omega}  ) }{n^{1/\beta } (\log \log n)^{-1/\beta } } =  c,
                    \qquad  \text{$P_x^{\omega}$-a.s. for all $x \in V(G^{\omega})$}.
   \end{equation} 
\end{thm}

\begin{proof}
  We follow the strategy in \cite{KKW}. 
    It is enough to prove that there exist positive constants $c_1, c_2>0$ such that the following holds, 
       \begin{equation}   \label{REinf} 
                c_1  \le \limsup_{ r \to \infty }   \frac{ \tau_{B(x,r)}^{\omega} }
                { r^{\beta} (\log \log r^{\beta})}    \le c_2, 
              \qquad   \text{$P_x^{\omega}$-a.s. for all $x \in V(G^{\omega} )$,}    
       \end{equation} 
    where $\tau_{B(x,r)}^{\omega} = \inf \{ n \ge 0 \mid X_n^{\omega} \not\in B(x,r) \}$. 
   Indeed, putting $n= r^{\beta} (\log \log r^{\beta})$ into \eqref{REinf} and using Theorem \ref{RE0-1}, we can easily obtain \eqref{REinf10}. 
    In the following, we use the notation in \eqref{Not10}.

\underline{Lower bound of \eqref{REinf};}     
 It is enough to show that there 
 exist constants $\eta>0$ and $J(\omega)>0$ such that 
      \begin{gather}   \label{REinf11}
           P_x^{\omega}   \left( \max_{ a_m \le r \le a_{2m} }    \frac{ \tau_{B(x,r)}^{\omega}   }{ r^{\beta} (\log \log r^{\beta}) }  \le \eta \right)  
                    \le    \exp ( - m^{1/4} )   
      \end{gather}
 holds for all $m \ge J(\omega )$, since the lower bound of \eqref{REinf} follows by \eqref{REinf11} and the Borel-Cantelli Lemma. 

First, we estimate the left hand side of \eqref{REinf11} as follows,  
 \begin{align}  \label{REinf12}
          &P_x^{\omega}  \left( \max_{2a_m \le r \le 2a_{2m}}   \frac{\tau_{B(x,r) }^{\omega}  }{r^{\beta} (\log \log r^{\beta}) }  \le \eta  \right)   
                  \le  P_x^{\omega}   \left(    \max_{m  \le  k  \le 2m}   \frac{ \tau_{B(x,2a_k)}^{\omega}   }{u_k }    \le 1   \right)      \notag \\ 
          & \le P_x^{\omega}    \left(    \max_{m  \le  k  \le 2m}   \frac{ \tau_{B(x,2a_k)}^{\omega}   }{  \sigma_k }    \le 1   \right)  
                \le P_x^{\omega} \left( \bigcap_{m \le k \le 2m} \left\{  \max_{0 \le s \le \sigma_{k+1} } d(X_0^{\omega},X_s^{\omega} ) \ge 2a_k \right\}  \right)  \notag \\
          &=  P_x^{\omega} (A_m^{\omega} )  , 
 \end{align}
  where we define $\displaystyle D_k^{\omega} = \left\{  \max_{0 \le s \le \sigma_{k+1} } d(X_0^{\omega},X_s^{\omega}) \ge 2a_k \right\}$ and use 
  $\displaystyle A_m^{\omega}   = \bigcap_{k=m}^{2m} D_k^{\omega}$ in the last equation.  
In order to estimate $P_x^{\omega}  (A_m^{\omega} )$, 
set  
 \begin{align*}
     G_k^{\omega}  &= \left\{ \max_{\sigma_k \le s \le \sigma_{k+1} }  d(X_{\sigma_k}^{\omega}  ,X_s^{\omega}  ) > a_k, d(X_0^{\omega} ,X_{\sigma_k}^{\omega}) <a_k \right\}, \\
     H_k^{\omega}  &= \left\{ \max_{0 \le s \le \sigma_k}  d(X_0^{\omega},X_{s}^{\omega}  ) \ge a_k \right\}. 
 \end{align*}
We can easily see $D_k^{\omega} \subset G_k^{\omega} \cup H_k^{\omega}$. 
Let $\eta \ge 1$ be as in Corollary \ref{REHK90}. 
For any $k$ with $\displaystyle \max_{z \in B(x,4\eta a_k)} N_{z,\epsilon}(\omega) \le a_k^{1/\beta}$,  we have 
          \begin{align*}  
               P_x^{\omega}  (G_k^{\omega})     
                  &  = E_x^{\omega} \left[ 1_{ \{ d(x,X_{\sigma_k}^{\omega} ) < a_k \} }
                                             P_{X_{\sigma_k}^{\omega} }^{\omega}  \left(  \max_{0 \le s \le u_k}  d(X_0^{\omega}, X_s^{\omega} ) >a_k \right) \right]   \\
              &  \le  \max_{z \in B^{\omega} (x,a_k) } 
              P_z^{\omega} \left( \max_{0 \le s \le u_k } d(z , X_s^{\omega})  > a_k  \right)\\   
               &  = 1 - \min_{z \in B^{\omega} (x,a_k) } 
               P_z^{\omega} \left( \max_{0 \le s \le u_k } d(z , X_s^{\omega})  \le a_k  \right)   \\
              &  \le  1-\frac{1}{ (1+k)^{2/3} }   \le   \exp \left( -c_3  k^{-2/3}  \right) ,
          \end{align*}
 where we use Corollary \ref{REHK90} in the forth inequality.  
 So, it holds that 
           \begin{align} \label{REinf15}
                      \max_{z\in B^{\omega} (x,a_k)}   P_z^{\omega}  (G_k^{\omega})    
                            \le   \exp \left( -c_3  k^{-2/3}  \right) 
           \end{align}
for any  $k$ with $\displaystyle \max_{z \in B(x,5\eta a_k)} N_{z,\epsilon}(\omega) \le a_k^{1/\beta}$.
 Hence, by Lemma \ref{Tail10}\,(3), \eqref{REinf15} holds for  $k \ge m \ge K_{x,\epsilon,1,5\eta} (\omega)$.  
 For any $k \ge m \ge L_{x,\epsilon ,2/3,1/3} (\omega)$ we have    
          \begin{align}  \label{REinf16}  
                    P_x^{\omega}   (H_k^{\omega} )   
                   &  \le c_4  \exp \left[ -c_{5}  \left(  \frac{a_k^{\beta}}{\sigma_k}  \right)^{1/(\beta-1)}  \right]    \notag    \\
                 &     \le c_{6}  \exp \left[  -c_{7}  \left( \frac{a_k^{\beta} }{ (k-1)  \lambda_{k-1}  a_{k-1}^{\beta} }  \right)^{1/(\beta-1)}  \right] \notag \\ 
                  &   \le c_{8}  \exp   \left[  -c_{9}  \left(  \frac{ e^{2k} }{k\log k} \right)^{1/(\beta -1)}     \right]  ,
           \end{align}
 where we use Lemma \ref{REHK20}\,(1) and Lemma \ref{Tail10}\,(1)  in the first inequality. 
 We can easily see 
        \begin{gather*} 
                   A_m^{\omega}    \subset \left( \bigcap_{k=m}^{2m} G_k^{\omega}  \right)   \cup   \left( \bigcup_{k=m}^{2m}  H_k^{\omega}  \right) .
        \end{gather*}
 Using the Markov property,  \eqref{REinf15} and \eqref{REinf16} we have 
        \begin{align}   \label{REinf17}
                 P_x^{\omega}   (A_m^{\omega} )  &\le  
                       \prod_{k=m}^{2m}  \exp (-c_3  k^{-2/3} ) 
                                 + c_{8} \sum_{k=m}^{2m} \exp   \left[ -c_{9}  \left( \frac{e^{2k}}{k \log k} \right)^{1/(\beta-1)} \right] \notag \\
                        &\le  \exp (- c_{10} m^{1/4} )     
        \end{align}  
 for any $m \ge K_{x,\epsilon,1,5\eta} (\omega) \vee L_{x,\epsilon,2/3,1/3} (\omega)$. 
   By \eqref{REinf12} and \eqref{REinf17} we obtain 
              \begin{align*}
                        \sum_{m }  P_x^{\omega}  \left( \max_{2a_m \le r \le 2a_{2m}}   \frac{\tau_{B(x,r) }^{\omega}  }{r^{\beta} (\log \log r^{\beta}) }  \le \eta  \right)   < \infty
              \end{align*}
  and thus by the Borel-Cantelli lemma, we obtain the lower bound of \eqref{REinf11}.

\underline{Upper bound;} 
 Define $\displaystyle B_k^{\omega} = \left\{ \max_{b_k \le r \le b_{k+1} }  \frac{ \tau_{B (x,r)}^{\omega} }{ r^{\beta} (\log \log r^{\beta})  }  \ge \eta  \right\}$.
 Then by Lemma \ref{REHK20}\,(2) and Lemma \ref{Tail10}\,(1), 
 for any $k \ge L_{x,\epsilon, 2,1}(\omega )$ we have 
                \begin{align*}
                           P_x^{\omega}  (B_k^{\omega} )    
                                  & \le    P_x^{\omega}  \left(   \tau_{B(x, b_{k+1} )}^{\omega}  \ge \eta b_k^{\beta} \log \log b_k^{\beta}  \right)   \\
                                 &  \le    P_x^{\omega}  \left(   \max_{0 \le s \le \eta b_k^{\beta}  \log \log b_k^{\beta} } d(X_0^{\omega},X_s^{\omega} )  \le b_{k+1} \right)    \\
                            &  =  P_x^{\omega}  \left(   \max_{0 \le s \le \frac{\eta}{e}  b_{k+1}^{\beta}  \log k}    d(X_0^{\omega},X_s^{\omega} )  \le b_{k+1} \right)    
                                        \le   \left( \frac{c_{11}}{k}  \right)^{c_{12} \eta/e}   .
                 \end{align*}
 Since the right hand side of the above is summable for sufficient large $\eta$,  by the Borel-Cantelli lemma we have 
                  \begin{align*}
                                \limsup_{k \to \infty}  \max_{b_k \le r \le b_{k+1}}
                                       \frac{ \tau_{B(x,r)}^{\omega}  }{ r^{\beta} (\log \log r^{\beta})  }  \le \eta ,   \qquad  \text{$P_x^{\omega}$-a.s.}
                  \end{align*}
 We can easily obtain the upper bound of \eqref{REinf} from the above inequality.  
 We thus complete the proof.
\end{proof}

\section{Ergodic media} \label{Zd}

In this section, we consider the case $G=(V,E) = \mathbb{Z}^d$ and obtain 
Theorem \ref{Thm20} under Assumption \ref{Ass} and Assumption \ref{Ass2}.  

\subsection{Ergodicity of the shift operator on $\Omega^{\mathbb{Z}}$}

Let $\Omega = [0,\infty )^{E }$ and define  $\mathscr{B}$ as the natural $\sigma$-algebra (generated by coordinate maps).  
We write  $\mathcal{X} = \Omega^{\mathbb{Z}}$, $\mathscr{X} = \mathscr{B}^{\otimes \mathbb{Z}}$ and denote a shift operator by $\tau_x$, i.e. $(\tau_x \omega )_e = \omega_{x+e}$.
If each conductance may take the value $0$, we regard $0$ as the base point and define 
$\mathcal{C}_0 (\omega) = \{ x \in \mathbb{Z}^d \mid 0 \overset{\omega}{\longleftrightarrow} x \}$, 
 where  $0 \overset{\omega}{\longleftrightarrow} x $ means that there exists a path $\gamma = e_1 e_2 \cdots e_k$ from $0$ to $x$ such that $\omega (e_i) > 0$ for all $i=1, 2,\cdots, k$. 
Define $\Omega_0 = \{ \omega \in \Omega \mid \sharp \mathcal{C}_0 (\omega) = \infty \}$ 
 and $\mathbb{P}_0 = \mathbb{P} (\cdot \mid \Omega_0)$.

Next we consider the Markov chain on the random environment (called the environment seen from the particle) 
according to Kipnis and Varadhan \cite{KV}.
Let $\omega_n ( \cdot ) = \omega ( \cdot + X_n^{\omega} ) = \tau_{X_n^{\omega} } \omega (\cdot ) \in \Omega$.  
We can regard this Markov chain $\{ \omega_n \}_{n \ge 0} $ as being defined on $\mathcal{X}  = \Omega^{\mathbb{Z}}$.  
We define a probability kernel $Q: \Omega_0 \times \mathscr{B} \to [0,1]$ as 
           \begin{align*}
                       Q(\omega , A) =    \frac{1}{ \sum_{e^{\prime} : |e^{\prime} | =1}   \omega_{e^{\prime}} }   
                                  \sum_{v: |v| = 1}  \omega_{0v}   1_{ \{ \tau_v \omega \in A \} }  .
           \end{align*}
This is nothing but the transition probability of the Markov chain $\{ \omega_n \}_{n \ge 0}$. 

Next we define the probability measure on $(\mathcal{X}, \mathscr{X})$ as 
           \begin{align*}
                       \mu \left(   (\omega_{-n} , \cdots, \omega_n) \in B   \right)   
                                   =  \int_B \mathbb{P}_0 (d\omega_{-n} ) Q(\omega_{-n} ,d\omega_{-n+1} )  \cdots Q(\omega_{n-1}, d\omega_n ) .
          \end{align*}
 By the above definition, $\{ \tau_{X_k^{\omega} } \omega \}_{k \ge 0}$ has the same law in 
$\mathbb{E}_0 ( P_0^{\omega} (\cdot ) )$ as $(\omega_0, \omega_1, \cdots)$ has in $\mu$, 
 that is, 
           \begin{align}  \label{SameDist}
                      \mathbb{E}_0 \left[ P_0^{\omega} (\{ \tau_{X_k^{\omega} } \omega \}_{k \ge 0} \in B ) \right]   =  \mu (  (\omega_0, \omega_1, \cdots) \in B )
           \end{align}
  holds for any $B \in \mathscr{X}$.

We need the following theorem to derive Theorem \ref{Thm20}. 
Let $T : \mathcal{X} \to \mathcal{X}$ be a shift operator of $\mathcal{X}$, that is, 
         \begin{align*}
                          (T \omega)_n =  \omega_{n+1}.
         \end{align*}

\begin{thm}\label{Thm:Erg}
Under Assumption \ref{Ass2}, $T$ is ergodic with respect to $\mu$.
\end{thm}
The proof is similar to \cite[Proposition 3.5]{BB}, so we omit it.

\subsection{The Zero-One law} 
The purpose of this subsection is to give the proof of Theorem $\ref{Thm20}$. 
We need the following version of the 0-1 law. 
    Let $a\ge 0$ and $A_1^{\omega} (a) , A_2^{\omega} (a) $ be the events 
                      \begin{align*}
                                 A_1^{\omega} (a)  &= \left\{ \limsup_{n\to \infty}  \frac{ d(X_0^{\omega}, X_n^{\omega} ) }{ n^{1/\beta}  (\log \log n)^{1-1/\beta} } > a \right\} ,   \\
                                 A_2^{\omega} (a)  &= \left\{ \liminf_{n \to \infty}   
 \frac{ \max_{0 \le k \le n} 
 d(X_0^{\omega}, X_k^{\omega} ) }{ n^{1/\beta}  (\log \log n)^{-1/\beta} } > a \right\}   .
                      \end{align*}
Define  
                     \begin{align*}
                               \tilde{A}_i (a) = \left\{ \omega \in \Omega \mid \text{  $A_i^{\omega} (a)$ holds for  $P_x^{\omega}$-a.s. and for all $x \in \mathcal{C}_0 (\omega)$} \right\} .
                      \end{align*}

\begin{prp}  \label{0-1-2}
 $\mathbb{P}_0( \tilde{A}_i (a))$ is either $0$ or $1$.
\end{prp}

\begin{proof}
We follow the proof of \cite[Corollary 3.2]{Copin}. 
Let $F_i : \Omega \to [0,1]$ be $F_i (\omega ) = P_0^{\omega} (A_i^{\omega} (a) )$. By the Markov property of $\{ \omega_n = \tau_{X_n^{\omega} } (\omega ) \}_n$ we have 
                  \begin{align*}
                                  P_0^{\omega}  ( A_i^{\omega} (a) \mid \mathcal{F}_n^{\omega} ) = F_i (\omega_n ),  
                  \end{align*} 
    where $\mathcal{F}_n^{\omega} = \sigma (X_k^{\omega} \mid k \le n)$. 
   So  $\{ F_i (\omega_n ) \}_n$ is $\mathcal{F}_n^{\omega}$-martingale. By the martingale convergence theorem we see  
                 \begin{align*}
                      F_i (\omega_n )   \to 1_{A_i^{\omega} (a) }   \qquad   \text{$P_0^{\omega}$-a.s.}
                 \end{align*}
  Therefore 
                   \begin{align*}
                                  \mathbb{E}_0 \left[ P_0^{\omega} \left(  \lim_{N \to \infty} \frac{1}{N} \sum_{n=0}^{N-1} F_i (\omega_n) = 1_{A_i^{\omega} (a)} \right) \right]   = 1 .
                   \end{align*}

   Next we define $\tilde{F}_i : \Omega^{\mathbb{Z}} \to [0,1] $ by $\tilde{F}_i (\bar{\omega} ) = F_i ( \bar{\omega}_0)$. Since $T$ is ergodic w.r.t. $\mu$, Birkhoff's ergodic theorem gives
                   \begin{align*}
                                \mu \left( \lim_{N \to \infty}  \frac{1}{N}   \sum_{n=0}^{N-1}  \tilde{F}_i \circ T^n   = \int \tilde{F}_i d \mu   \right) =1.
                   \end{align*}
   By \eqref{SameDist} we see
                   \begin{align*}
                                   1_{A_i^{\omega} (a) }   =     \int \tilde{F}_i d \mu .
                   \end{align*}
So, either $A_i^{\omega} (a)$ holds almost surely or it does not hold almost surely.  
We thus complete the proof. 
\end{proof}

Theorem \ref{Thm10} and Proposition \ref{0-1-2} immediately give Theorem \ref{Thm20}.

\end{document}